\documentclass[11pt]{amsart}

\newcommand{\ca}{{\mathcal{K}}_0^n}

\newcommand{\ds}{\displaystyle}

\newcommand{\ii}{\infty}

\newcommand{\fp}{\left(\frac{h_K}{h_L}\right)}
\newcommand{\intl}{\int_{\mathbb{S}^{n-1}}}
\newcommand{\s}{\mathbb{S}^{n-1}}
\newcommand{\dm}{d\mu_{\mathbb{S}^{n-1}}}

\newtheorem{theorem}{Theorem}[section]
\newtheorem{proposition}{Proposition}[section]

\newtheorem{corollary}{Corollary}[section]

\usepackage{color}

\begin{document}

\title[A note on Bourgain-Milman's universal constant]
 {A note on Bourgain-Milman's universal constant}

\author[A. Stancu]{Alina Stancu${}^{1}$}
\address{Department of Mathematics and Statistics,
  Concordia University, Montreal, QC, Canada, H3G 1M8}
\email{alina.stancu@concordia.ca}
\subjclass[2010]{Primary 52A40; Secondary 52A20.}

\thanks{${}^{1}$ Partially supported
by an NSERC Discovery grant.}

\dedicatory{}

\begin{abstract}
The present note is a result of an on-going investigation into the logarithmic Brunn-Minkowski inequality. We obtain lower estimates on the volume product for convex bodies in $\mathbb{R}^n$ not necessarily symmetric with respect to the origin from a modified logarithmic Brunn-Minkowski inequality.
\end{abstract}

\maketitle

\section{Introduction}

With extraordinary implications which continue to be seen, the classical Brunn-Minkowski theory of convex bodies was placed in a larger theory by Lutwak's $L_p$-Minkowski problem \cite{Lutwak1, Lutwak2}. Consequently, many classical results for convex bodies became part of the extended $L_p$ Brunn-Minkowski-Firey theory, while many other results of the extended theory bring new and original insight in convex geometric analysis. One such strikingly new behavior is due to the $L_0$-Minkowski problem \cite{BLYZ1}, \cite{Stancu1}--\cite{Stancu3}, and its version of the Brunn-Minkowski inequality known for technical reasons as the logarithmic Brunn-Minkowski inequality \cite{BLYZ2}. The present note is a result of an on-going investigation into the logarithmic Brunn-Minkowski inequality.

Proved for $n=2$ by B\"or\"oczky, Lutwak, Yang and Zhang \cite{BLYZ2}, and conjectured by them for all $n$'s where certain cases are known to hold \cite{Saroglou}, \cite{Stancu4}, the logarithmic Brunn-Minkowski inequality states the following.

\bigskip

\noindent {\bf{The Logarithmic Brunn-Minkowski Inequality.}} {\em
Let $K$ and $ L$ be convex bodies in $\mathbb{R}^n$, centrally symmetric with respect to the origin. Then the following inequality holds
\begin{equation}
\int_{\mathbb{S}^{n-1}} \ln \left( \frac{h_K}{h_L} \right)\, d\bar{v} \geq\frac{1}{n}\ln \left( \frac{Vol (K)}{Vol (L)} \right), \label{eq:LBMI}\end{equation}
where $d{v}_L$ is the cone-volume measure of $L$ and  $ d\bar{v}_L=\frac{1}{ Vol (L)}\, d v_L$ is its normalization. \label{prop:LBMI} }

\bigskip

Our first result represents a modified logarithmic Brunn-Minkowski:

\begin{proposition}
Let $K$ and $L$ be convex bodies in $\mathbb{R}^n$ containing the origin in their interior. Then
\begin{equation}
\int_{\mathbb{S}^{n-1}} \ln \left( \frac{h_K}{h_L} \right)\, d\bar{v}_1 \geq \ln \left( \frac{V_1(L,K))}{Vol (L)} \right) \geq \frac{1}{n} \ln \left( \frac{Vol (K))}{Vol (L)} \right), \label{eq:MLBMI}
\end{equation}
where $d{v}_1$ is a mixed volume measure, namely $\ds \int_{\s} d{v}_1 =V( [n-1], 1, L, K) =:  V_1(L, K)$, and $\ds d\bar{v}_1=\frac{1}{ V_1( L,K)}\, d v_1$ is its normalization. \label{prop:MLBMI} Equality holds if and only if $K$ is homothetic to $L$.
\end{proposition}

 The second inequality of (\ref{eq:MLBMI}) follows immediately from a well-known inequality  for mixed volumes due to Minkowski and we included it due to the similarity with the conjectured one.
The main focus of this paper is on the implications of the first inequality of (\ref{eq:MLBMI}) for the lower bound of the volume product functional and for the Bourgain-Milman constant well-known from the following theorem.

\medskip

\noindent {\bf{Theorem (Bourgain-Milman)}} \cite{Bourgain} {\em There exists an absolute constant $c>0$ (thus independent of the
dimension $n$) such that, for any centrally symmetric convex body $K$,}
\begin{equation} Vol (K) \cdot Vol (K^\circ) > c^n \omega_n^2, \label{eq:BMC} \end{equation}
  {\em where $K^{\circ}$ is the polar of $K$ with respect to the origin and $\omega_n$ is the volume of the unit ball ${\mathbb{B}}_2^n$ in $\mathbb{R}^n$.}
\medskip

In this paper, we show:
\begin{theorem}
For any $K$ convex body in $\mathbb{R}^n$ containing the origin in its interior, we have \begin{equation} Vol (K)
\cdot Vol (K^\circ) > \max \left\{ evr^n(K),  evr^n(K^\circ) \right\}\, \omega_n^2,
\end{equation} where $evr(K)$ (and $evr(K^\circ)$) is the exterior volume ratio of $K$ (respectively $K^\circ$). \label{theorem:const_vr} In particular, if $K$ is symmetric about the origin,
\begin{equation} Vol (K)
\cdot Vol (K^\circ) > \left[ \frac{2^n \Gamma (\frac{n}{2}+1 )}{n!\, \pi^{n/2}}\right]\, \omega_n^2,
\end{equation}
while, for arbitrary convex bodies $K$, we have that \begin{equation} Vol (K)
\cdot Vol (K^\circ) > \left[\frac{(n+1)^{\frac{n+1}{2}} \Gamma(\frac{n}{2}+1)}{n! (n\pi)^{n/2}} \right]\, \omega_n^2.
\end{equation}
\end{theorem}

 The actual lower bound of the volume product is the subject of Mahler conjecture which, although supported by an impressive body of work, remains, except for some cases, still open.

\section{Results and Proofs} A convex body in $\mathbb{R}^n$ is a compact convex set in $\mathbb{R}^n$.
Let  $\ca$ be the set of convex bodies in $\mathbb{R}^n$ containing the origin in their interior.

For $K \in \ca$, we denote by $K^\circ =\{ x \in \mathbb{R}^n : x \cdot y \leq 1, \forall y \in K\}$ the polar body of $K$, where $x \cdot y$ is the standard inner product of $x$ and $y$ in $\mathbb{R}^n$.
Moreover, for $K \in \ca$, consider $h_K : \mathbb{S}^{n-1} \to \mathbb{R}_+$ the support function of $K$ defined by $$h_K(u)= \max \{ x \cdot u, \forall x \in K\},\  \ \  u \in \mathbb{R}^n,\ ||u||=1.$$

 The volume of a convex body $K$ is the $n$-dimensional Lebesgue measure of $K$ as a compact set of $\mathbb{R}^n$. We call $dS_L$ the surface area measure of $L$, $dv_L=\frac{1}{n}\,h_L dS_L$ the cone-volume measure of $L$, and $d{v}_1=\frac{1}{n}\, h_K \, dS_L$ the mixed volume measure of $K$ and $L$, as measures on $\mathbb{S}^{n-1}$, see \cite{Schneider} for a detailed discussion.

 Let ${\mathcal{F}}_n$ denote the convex bodies in $\mathbb{R}^n$ with positive continuous curvature function, and let ${\mathcal{V}}_n$ denote the convex bodies in $\mathbb{R}^n$ of elliptic type, \cite{Lei}. If $L$ is convex body in ${\mathcal{F}}_n$ with curvature function $f_L$, then $\ds \Omega (L)=\intl f_L^{\frac{n}{n+1}}\, \dm$ is the affine surface area of $L$, \cite{Lei}.

 Finally, if $K \subset {\mathbb{R}}^n$ is a convex body, the exterior volume ratio of $K$ is, by definition,
\begin{equation}
evr (K)= \left(\frac{Vol (K)}{Vol ({\mathcal{E}}_L)}\right)^{1/n}= \max_{K \subseteq {\mathcal{E}}}\left(\frac{Vol (K)}{Vol ({\mathcal{E}})}\right)^{1/n}, \label{eq:evr}
\end{equation}
where the maximum is taken after all ellipsoids containing $K$, see for example Ball's \cite{Ball} or Barthe's \cite{Barthe}. The ellipsoid ${\mathcal{E}}_L$ is called L\"owner's ellipsoid.

John's theorem \cite{John} states that, for any convex body $K$, its L\"owner's ellipsoid satisfies the inclusions $\frac{1}{n} \, {\mathcal{E}}_L \subseteq K \subseteq {\mathcal{E}}_L$, where $\frac{1}{n} \, {\mathcal{E}}_L$, generally denoted by ${\mathcal{E}}_J$, is also called John's ellipsoid, the ellipsoid of maximal volume contained in $K$.

 \medskip

 \noindent {\em{Proof of Proposition \ref{prop:MLBMI}.}} { Note that $$\ds \int_{\s}\frac{h_K}{h_L} \ln \left( \frac{h_K}{h_L} \right)\, d{v}_L = \intl \ln \fp \, dv_1.$$ Then the first claim follows from \begin{equation}
\exp \left[-\frac{n}{V_1( L, K)}\, \int_{\s}\frac{h_K}{h_L} \ln \left( \frac{h_K}{h_L} \right)\, d{v}_L \right] =\lim_{p \to \ii} \left[\frac{1}{V_1( L,K)}\,  \int_{\s}\left(\frac{h_K}{h_L} \right)^{\frac{p}{p+n}} dv_L \right]^{p+n} \nonumber
\end{equation} and H\"older's inequality
\begin{equation}
\left( \int_{\s}\left(\frac{h_K}{h_L} \right)^{\frac{p}{p+n}} dv_L \right)^{\frac{p+n}{p}} \cdot \left( \int_{\s} dv_L \right)^{-\frac{n}{p}} \leq \int_{\s} \frac{h_k}{h_L}\, dv_L= V_1( L, K), \nonumber
\end{equation}
where $\ds \int_{\s} dv_L= Vol (L)$.}

 By applying Minkowski's inequality  $V_1( L, K) \geq Vol (K)^{1/n} \cdot Vol (L)^{(n-1)/n}$ to the first inequality of Proposition \ref{prop:MLBMI}, we obtain the second inequality and that
 \begin{equation}
 \int_{\mathbb{S}^{n-1}} \ln \left( \frac{h_K}{h_L} \right)\, d\bar{v}_1 \geq \frac{1}{n} \ln \left( \frac{Vol(K))}{Vol (L)} \right).
 \end{equation}

 \hfill $\Box.$

The first inequality in (\ref{eq:MLBMI})  is a strengthened version of an inequality obtained by Gardner-Hug-Weil \cite{Gardner}.
 They showed in the larger set-up of the Orlicz-Brunn-Minkowski theory that, if $L, K \in \ca$ and $L \subseteq K$, the following inequality holds
\begin{equation}
\int_{{\mathbb{S}}^{n-1}}  \frac{h_K}{h_L}\, \ln \left( \frac{h_K}{h_L} \right)\, d\bar{v}_L \geq \frac{1}{n} \left( \frac{Vol (K)}{Vol (L)} \right)^{1/n} \ln \left( \frac{Vol (K)}{Vol (L)} \right). \label{eq:Gardner}
\end{equation}
More precisely, the last inequality follows from Lemma 9.1 of  \cite{Gardner} by considering the convex function $x \mapsto  x \ln x$.

Our first inequality in (\ref{eq:MLBMI}) can be written as \begin{equation}
\int_{{\mathbb{S}}^{n-1}}  \frac{h_K}{h_L}\, \ln \left( \frac{h_K}{h_L} \right)\, d\bar{v}_L \geq \frac{V_1(L, K)}{Vol (L)} \ln \left( \frac{V_1(L, K)}{Vol (L)} \right), \label{eq:Gardner2}
\end{equation} from which $L \subseteq K$ and Minkowski`s inequality implies (\ref{eq:Gardner}).

\medskip

In connection with the conjectured inequality, remark  the following fact from information theory: {\em
If $p, q$ are probability density functions on a measure space $(X, \nu)$, then \begin{equation} \int p \ln p \, d\nu \geq \int p \ln q \, d\nu. \end{equation}}
By taking $\ds p\, d \nu = \frac{h_L}{h_K} \cdot \frac{1}{Vol (L)}\, dv_1$ and $\ds q\, d \nu = \frac{1}{V_1(L, K)} \, dv_1$ (also switching the places of the two measures), we obtain the double inequality
\begin{equation}
\intl \ln \fp d\bar{v}_L \leq \ln \left( \frac{V_1( L, K)}{Vol (L)} \right) \leq \intl \ln \fp {d\bar{v}_1}.\label{eq:chain}
\end{equation}

\medskip

The next proposition is fundamental for our main theorems.

\begin{proposition}
For any $K \in \ca$ and any $L \in \ca \cap {\mathcal{F}}_n$, we have \begin{equation}
Vol (K) \cdot Vol (K^\circ)\geq \frac{1}{n^{n+1}}\cdot \frac{\Omega^{n+1}(L)}{Vol^{n-1} (L)} \cdot \frac{\frac{Vol (K)}{Vol (L)}}{\left[ \exp \int_{\mathbb{S}^{n-1}} \ln \left(\frac{h_K}{h_L}\right)\, d\bar{v}_1 \right]^n}, \label{eq:sur}
\end{equation} with equality if and only if $K$ and $L$ are homothetic ellipsoids centered at the origin. \label{proposition:sur}
\end{proposition}

\begin{proof}
Suppose first that $K$ and $L$ are distinct. Rewrite the inequality of Proposition \ref{prop:MLBMI}, as \begin{equation}
\exp \int_{\mathbb{S}^{n-1}} \ln \left( \frac{h_K}{h_L} \right)\, d\bar{v}_1 \geq \frac{V_1(L,K)}{Vol (L)}=\frac{1}{Vol (L)} \frac{1}{n} \int_{\mathbb{S}^{n-1}}h_K f_L \, \dm, \nonumber
\end{equation} where $f_L$ is the curvature function of $\partial L$ as a function on the unit sphere. Since \begin{equation} \int_{\mathbb{S}^{n-1}}h_K f_L \, \dm \geq \left(\intl h_K^{-n}\dm\right)^{-\frac{1}{n}}\left(\intl f_L^{\frac{n}{n+1}}\dm \right)^{\frac{n+1}{n}} \label{eq:rH} \end{equation} we obtain
\begin{equation}\exp \int_{\mathbb{S}^{n-1}} \ln \left( \frac{h_K}{h_L} \right)\, d\bar{v}_1 \geq \frac{(\Omega (L))^{\frac{n+1}{n}}}{ n^{1+\frac{1}{n}}Vol (L)} \left[\frac{1}{n}\intl \rho_{K^\circ}^{n}\dm\right]^{-\frac{1}{n}},
\end{equation} where $\Omega (L)$ is the affine surface area of $L$ mentioned earlier.
Thus,  by raising both sides of the previous inequality to the power $n$, we have \begin{equation}\left[\exp \int_{\mathbb{S}^{n-1}} \ln \left( \frac{h_K}{h_L} \right)\, d\bar{v}_1\right]^n \geq \frac{(\Omega^{n+1} (L))}{ n^{n+1}Vol^n (L)} \frac{1}{Vol (K^{\circ})},
\end{equation} from which a rearrangement of terms yields the claim.

If $K=L \in \ca \cap {\mathcal{F}}_n$, the claim follows directly from the reverse H\"older inequality (\ref{eq:rH}).

Finally, note that the equality holds in (\ref{eq:rH}) if and only if $L$ is of elliptic type and $K$ is its curvature image body (up to a constant), $h_K=\lambda f_L^{-\frac{1}{n+1}}$, for some positive constant $\lambda$. As, in addition, equality holds in Proposition \ref{prop:MLBMI} if and only if $K$ is homothetic to $L$, we conclude, due to Petty's lemma, that equality holds in (\ref{eq:sur}) if and only if $K$ and $L$ are homothetic ellipsoids centered at the origin.

\end{proof}

\medskip

\noindent {\em Proof of Theorem 1.1.} Since the left-hand side is invariant under linear transformations, apply to $K$ the linear transformation which transforms its L\"owner's ellipsoid into the unit ball ${\mathbb{B}}_2^n$. For simplicity, keep the same notation for $K$ and apply Proposition \ref{proposition:sur} to $K$ and ${\mathcal{E}}_L$, obtaining, since $h_K \leq h_{{\mathcal{E}}_L}$ in all directions,
\begin{equation}
Vol (K) \cdot Vol (K^\circ)\geq evr^n(K)\cdot \omega^2_n.
\end{equation}

Of course, the role of $K$ and $K^\circ$ can be interchanged and we also have $Vol (K) \cdot Vol (K^\circ)\geq evr^n(K^\circ)\cdot \omega^2_n$. Therefore, as remarked to us by V. Milman, it is desirable to write explictly the maximum of the two exterior volume ratios as, for example, for the cross-polytope the exterior volume ratio is extremely small, while for the cube is uniformly bounded from below.

The second part of the claim follows now from the following results due to Barthe, \cite{Barthe}, see also Ball \cite{Ball}:
For any convex body $K \subset {\mathbb{R}}^n$, we have $evr(K) \geq evr (\Delta_n)$, where $\Delta_n$ is the regular simplex of ${\mathbb{R}}^n$ inscribed in the Euclidean ball ${\mathbb{B}}_2^n$. On the other hand, for any convex body $K \subset {\mathbb{R}}^n$ symmetric with respect to the origin, we have $evr(K) \geq evr ({\mathbb{B}}_1^n)$, where ${\mathbb{B}}_1^n$ is here the $L^1$ ball inscribed in the Euclidean ball, for which the calculations follow immediately.

In the case of the simplex, the calculations for the exterior volume ratio are in Satz 4 of \cite{Jun} as pointed out to the author by S. Taschuk.

\hfill $\Box$

\medskip

The next corollary is a special case of Proposition \ref{proposition:sur} which we feel justified to state  separately.

\begin{corollary}
For any $K$ convex body in $\ca \cap {\mathcal{F}}_n$, we have \begin{equation}
Vol (K) \cdot Vol (K^\circ)\geq  \frac{1}{n^{n+1}}\cdot \frac{\Omega^{n+1}(K)}{Vol^{n-1} (K)},
\end{equation} with equality if and only if $K$ is a centered ellipsoid. \label{corollary:main}
\end{corollary}


\

Note that Corollary \ref{corollary:main}, together with Santal\'o inequality, implies  the famous affine isoperimetric inequality  for bodies in $\ca \cap {\mathcal{F}}_n$. The  affine isoperimetric inequality states that {\em for any convex body} $K$
\begin{equation}\frac{\Omega^{n+1}(K)}{Vol^{n-1} (K)} \leq n^{n+1} \omega_n^2, \ \ \ {\hbox{where}}\ \ \omega_n=Vol (\mathbb{B}^{2}_n),\end{equation}
see for example \cite{Lei}.
\hfill $\Box.$

\medskip

We will use now Proposition \ref{proposition:sur} for the special case of $L$ the unit ball in $\mathbb{R}^n$. As before,  $\omega_n$ stands for the volume of the unit ball in $\mathbb{R}^n$, and $\ds w(K)=\intl h_K\, d\mu_{\mathbb{S}^{n-1}}$ is, up to the factor of $1/2$, the mean width of $K$. By using the concavity of $x \mapsto \ln x, \ x>0$, and the fact that the volume product is invariant under (special) linear transformations, we obtain:

\begin{corollary}
For any $K$ convex body in $\ca$, we have \begin{equation}
Vol (K) \cdot Vol (K^\circ)\geq
 \omega_n \cdot \max_{T \in SL(n)}\frac{Vol (K) w^n(TK)}{\left( \int_{\mathbb{S}^{n-1}}  h^2_{TK}\, d\mu_{\mathbb{S}^{n-1}}\right)^n}, \label{eq:max}
\end{equation} where $K$ and $K^\circ$ can be interchanged.  Equality occurs above if and only if $K$ is an ellipsoid centered at the origin.\\
 Thus, for any $K \in \ca$, \begin{equation}
Vol (K) \cdot Vol (K^\circ)\geq \omega_n \min_{L \in \ca} \max_{T \in SL(n)}\frac{Vol (L) w^n(TL)}{ (\int_{\mathbb{S}^{n-1}}  h^2_{TL}\, d\mu_{\mathbb{S}^{n-1}})^n}.
\end{equation} \label{cor:M}
\end{corollary}

\hfill $\Box.$
\medskip

It is easy to check that if $\{K_j\}_j$ is a sequence of centered ellipsoids with equal volume whose major axis goes to infinity as $j \nearrow +\infty$, then $\ds \frac{Vol (K_j) w^n(K_j)}{\left( \int_{\mathbb{S}^{n-1}}  h^2_{K_j}\, d\mu_{\mathbb{S}^{n-1}}\right)^n} \searrow 0$. Thus, taking the maximum after all $SL(n)$ transformations in (\ref{eq:max}) is needed.  In fact, we believe that, for a given convex body $K$ in $\ca$, the optimal value $\ds M(K):=  \max_{T \in SL(n)}\frac{Vol (K) w^n(TK)}{\left( \int_{\mathbb{S}^{n-1}}  h^2_{TK}\, d\mu_{\mathbb{S}^{n-1}}\right)^n}$ is reached for some isotropic position of $K$ in the sense of Giannopoulos-Milman, \cite{GM}.


\medskip



\begin{thebibliography}{99}

\bibitem{Ball} K. Ball, {\em Volume ratios and a reverse isoperimetric inequality},
J. London Math. Soc. 44 (1991), 351--359.

\bibitem{Barthe} F. Barthe, {\em Autour de l'in\'egalit\'e de Brunn-Minkowski}, Ann. fac. sci. math. de Toulouse, Vol XII (2003), 127--178.

\bibitem{BLYZ1} K. B\"or\"oczky; E. Lutwak; D. Yang; G. Zhang, {\em The log-Brunn-Minkowski inequality}, Adv. Math.  {{231}}  (2012), 1974-–1997.

\bibitem{BLYZ2} K. B\"or\"oczky; E. Lutwak; D. Yang; G. Zhang, {\em The logarithmic Minkowski problem}, J. Amer. Math. Soc.  {{26}}  (2013), 831–-852.

 \bibitem{Bourgain}   J. Bourgain; V. D.Milman, {\em New volume ratio properties
for convex symmetric bodies in $\mathbb{R}^n$}, Invent. Math. {{88}} (1987), 319–-340.




\bibitem{Gardner}  R.J. Gardner; D. Hug; W. Weil, {\em The Orlicz-Brunn-Minkowski theory: a general framework, additions, and inequalities}, to appear in J. Differential Geom.

 \bibitem{Gian}   A. Giannopoulos; G. Paouris; B.-H. Vritsiou, {\em The isotropic position and the reverse Santal\'o inequality}, to appear in Israel J. Math.

     \bibitem{GM}   A. Giannopoulos; V.D. Milman, {\em Extremal problems and isotropic positions of convex bodies}, Israel J. Math. {{117}} (2000), 29--60.

  \bibitem{Gian2}   A. Giannopoulos; V.D. Milman; M. Rudelson, {\em Convex bodies with minimal mean width},
Geometric Aspects of Functional Analysis,
Lecture Notes in Math. {{1745}} (2000),  81--93,  Springer, Berlin Heidelberg.

 \bibitem{HLYZ} C. Haberl; E. Lutwak; D. Yang; G. Zhang, {\em The even Orlicz Minkowski problem}, Adv. Math.  224  (2010), 2485--2510.

\bibitem{John} F. John, {\em Extremum problems with inequalities as subsidiary conditions},  Studies and Essays Presented to R. Courant on his 60th Birthday, January 8, 1948,  187--204. Interscience Publishers, Inc., New York, N. Y. (1948).


 \bibitem{Jun}  F.  Juhnke, {\em Volumenminimale Ellipsoidd\"uberdeckungen},
    Beitr\"age Algebra Geom. {{50}} (1990), 150--154.

    \bibitem{Kuperberg} G. Kuperberg,
{\em From the Mahler conjecture to Gauss linking integrals}, Geom. Funct. Anal.  {{18}}  (2008), 870-–892.


\bibitem{Lei} K. Leichtweiss, {\em Affine Geometry of Convex Bodies}, Johann Ambrosius Barth Verlag, Heidelberg, Leipzig (1998).

\bibitem{Lutwak1} E. Lutwak, {\em The Brunn-Minkowski-Firey theory. I: Mixed volumes and the Minkowski problem}, J. Differential
Geom. {{38}} (1993), 131-–150.

\bibitem{Lutwak2}
{ E. Lutwak}, {\em The {B}runn-{M}inkowski-{F}irey theory {II}:
Affine and geominimal
  surface areas},
Adv. Math. {118} (1996), 244--294.



  \bibitem{Nazarov} F.  Nazarov, {\em The H\"ormander proof of the Bourgain-Milman theorem},  Geometric aspects of functional analysis,  335--343, Lecture Notes in Math., 2050 (2012), Springer, Heidelberg.

\bibitem{Saroglou} C. Saroglou, {\em Remarks on the conjectured
log-Brunn-Minkowski inequality}, preprint 2013, submitted.


\bibitem{Schneider} R. Schneider,
{\em Convex bodies: The Brunn-Minkowski theory}, Cambridge Univ.
Press, New York (1993).


\bibitem{Stancu1} A. Stancu, {\em The discrete planar
${\hbox{L}}_0$-Minkowski problem},
Adv. Math., Vol. 167 (2002), 160--174.

\bibitem{Stancu2} A. Stancu, {\em On the number of solutions to the
discrete two dimensional ${\hbox{L}}_0$-Minkowski problem}, Adv.
Math., Vol. 180 (2003), 290--323.

\bibitem{Stancu3} A. Stancu, {\em  The necessary condition for the discrete
$L_0$-problem in $\mathbb{R}^2$}, J. of Geom., Vol. 88 (2008),
162--168.


\bibitem{Stancu4}  A. Stancu, {\em Logarithmic Brunn-Minkowski inequality for non-symmetric convex bodies}, in preparation.



\end{thebibliography}
\end{document}